\documentclass[12pt]{compositio}

\usepackage{latexsym, amssymb, amsmath}

\usepackage{amsfonts, graphicx}

\newcommand{\be}{\begin{equation}}
\newcommand{\ee}{\end{equation}}
\newcommand{\beq}{\begin{eqnarray}}
\newcommand{\eeq}{\end{eqnarray}}

\newtheorem{thm}{Theorem}[section]

\newtheorem{lma}{Lemma}[section]
\newtheorem{prop}{Proposition}[section]

\theoremstyle{remark}
\newtheorem{rem}{Remark}[section]
\numberwithin{equation}{section}

\def\p{\partial}

\def\R{\mathbb{R}}

\def\tr{{\rm tr}}
\def\p{\partial}
\def\lf{\left}
\def\ri{\right}
\def\e{\epsilon}

\def\R{\Bbb R}

\def\Ric{\text{\rm Ric}}

\def\Pi{\overline{\displaystyle{\mathbb{II}}}}

\def\a{\alpha}
\def\b{\beta}

\def\ii{\sqrt{-1}}

\def\heat{\lf(\Delta -\frac{\p}{\p t}\ri)}
\def\heatt{\lf(\frac{\p}{\p t}-\Delta\ri)}
\def\K{K\"ahler }

\def\ddbar{\partial\bar{\partial}}
\def\aint{\frac{\ \ }{\ \ }{\hskip -0.4cm}\int}
 \def\be{\begin{equation}}
 \def\ee{\end{equation}}
\def\bee{\begin{equation*}}
 \def\eee{\end{equation*}}

\begin{document}

\title[Poincar\'e-Lelong equation]{\bf Poincar\'e-Lelong equation via the Hodge Laplace heat equation}

\author{ Lei Ni}
\email{lni@math.ucsd.edu}

\address{Department of Mathematics\\
         University of California at San Diego\\
         La Jolla, CA 92093}

\author{ Luen-Fai Tam}
\email{lftam@math.cuhk.edu.hk}

\address{The Institute of Mathematical Sciences and Department of
 Mathematics\\
The Chinese University of Hong Kong\\
Shatin, Hong Kong, China}


\classification{ 53C21 (primary), 32Q15, 35B40, 35K05(secondary).}

\keywords{Poincar\'e-Lelong equation, Hodge-Laplacian heat equation, K\"ahler manifolds, Convex exhaustion.  }

\thanks{The first author is partially supported by NSF grant DMS-1105549.  The second author is partially supported by  Hong Kong RGC General Research Fund
\#CUHK 403011}

\begin{abstract} In this paper, we develop a method of solving the Poincar\'e-Lelong equation, mainly via the study of the large time asymptotics of a global solution to the Hodge-Laplace heat equation on $(1, 1)$-forms. The method is effective in proving an optimal result when $M$ has nonnegative bisectional curvature. It also provides an alternate proof of a recent gap theorem of the first author.
\end{abstract}

\maketitle

\section{Introduction}

Solving the Poincar\'e-Lelong equation amounts to,  for a given real $(1,1)$-form $\rho$,  finding a smooth function $u$ such that $\sqrt{-1}\partial \bar{\partial} u=\rho$. Motivated by geometric considerations, on a complete noncompact K\"ahler manifold $(M, g)$, this was first studied by Mok et al. \cite{MSY} under some restrictive conditions including a point-wise quadratic decay on $\|\rho\|$, nonnegative bisectional curvature and maximum volume growth on $M$. There have been many publications since then (e.g. \cite{Ni-98}, \cite{NST}, \cite{Fan2007}). Finally in \cite{NT-2003}, the following result was proved.
\begin{thm}\label{pl}
Let $M^n$ (with $m=2n$ being the real dimension) be a complete K\"ahler manifold
with nonnegative holomorphic bisectional curvature.  Let $\rho$ be
a real $d$-closed $(1,1)$-form. Suppose that
\begin{equation}\label{ass1}
 \int_0^\infty  \left(\aint_{B_o(s)}||\rho||(y)d\mu(y)\right) ds<\infty,
\end{equation}
 and
 \begin{equation}\label{extr2}
\liminf_{r\to\infty}\lf[\exp\lf(-\a r^2\ri)\cdot\int_{B_o(r)}||\rho||^2(y) d\mu(y)\ri]<\infty
\end{equation}
for some $\a>0$.
 Then there is a solution $u$ of
the Poincar\'e-Lelong equation $$\sqrt{-1}\ddbar u =\rho.$$
Moreover, for any $0<\epsilon<1$, $u$ satisfies
\begin{eqnarray}\label{bound}
  \alpha_1r\int_{2 r}^\infty k(s)ds&+&\beta_1\int_0^{2r}sk(s)ds
 \ge u(x)\nonumber\\
 &\ge& \beta_3\int_{0}^{2r}sk(s)ds
 -\alpha_2r\int_{2r}^\infty k(s)ds-\beta_2\int_0^{\epsilon
r}sk(x,s)ds
\end{eqnarray}
for some positive constants $\alpha_1(m)$, $\alpha_2(m,\epsilon)$
and $\beta_i(m)$, $1\le i\le 3$, where $r=r(x)$. Here  $k(x,s)= \aint_{B_x(s)}\|\rho\|$ and $k(s)=k(o,s)$, where $o\in M$ is  fixed.
\end{thm}
 Due to the technical nature of the assumption (\ref{extr2}), which arises from the parabolic method employed in \cite{NT-2003},  and is related to the uniqueness of the heat equation, it is desirable to be able to remove it.
The purpose of this paper is to prove:

 \begin{thm}
Let $(M^n,g)$ be a complete noncompact \K manifold (of complex dimension $n$) with nonnegative Ricci curvature and nonnegative quadratic orthogonal bisectional curvature.    Suppose that $\rho$ is a smooth $d$-closed   real $(1,1)$-form on $M$ and let $f=||\rho||$ be the norm of $\rho$.  Suppose that
\be\label{eq-rho decay}
\int_0^\infty k_f(r)dr<\infty.
\ee
Then there is a smooth function $u$ so that $\rho=\ii \p\bar\p u$. Moreover, for any $0<\e<1$, the estimate (\ref{bound}) holds.
\end{thm}

  Note that in the theorem  we only require $(M^n, g)$ has nonnegative Ricci curvature and nonnegative quadratic orthogonal bisectional curvature (see Section 3 for the definition), which is weaker than the nonnegativity of the bisectional curvature (see also Section 7 for the discussion on the relations between various curvature conditions). The current result is also   more general than those in a previous version of our work and the proof is less involved.

The solution space to a Poincar\'e-Lelong equation clearly is an affine space consisting  $u_s+u_h$ with $u_s$ being a special solution and $u_h$ being any element of the linear space of the pluriharmonic functions. The estimate (\ref{bound}) selects the minimum one among them. In the view that on K\"ahler manifolds with nonnegative Ricci curvature, the sublinear growth is  the optimal necessary condition to imply  the constancy of a pluriharmonic function, the assumption (\ref{ass1}) is almost the optimal condition which  one can expect to ensure that  estimate (\ref{bound}) selects the unique (up to a constant) solution.

 The use of the Hodge-Laplace heat equation here is motivated by that of \cite{PL-remove}. On the other hand here we establish the closedness of the global $(1,1)$-form solution to the Hodge-Laplace heat equation,  and in Section 5  obtain an alternate proof to the gap theorem in \cite{PL-remove} without appealing  to the relative monotonicity.

\begin{acknowledgements} The authors would   like to thank Gilles Carron for pointing out a mistake in a previous version of this paper. The authors   would  also like to thank Alexander Grigoryan for useful discussions.
\end{acknowledgements}

\section{A general method of solving the Poincar\'e-Lelong equation}

Let $(M^n,g)$ be a complete noncompact K\"ahler manifold with complex dimension $n$ with \K form. Let $\rho$ be a  real (1,1) form. If in a local coordinate $\rho=\sqrt{-1}\rho_{i\bar{j}}dz^i\wedge dz^{\bar{j}}$ we denote the  trace of $\rho$, $\sum_{i, j=1}^n g^{i\bar{j}} \rho_{i\bar{j}}$  by $\tr(\rho)$ and is equal to $\Lambda \rho$, where $\Lambda$ is the adjoint of $L$ with $L\sigma=\omega\wedge\sigma$.

Let $\Delta=-\Delta_d=-(d\delta+\delta d)$ be the Hodge Laplacian for forms. On K\"ahler manifolds it is well known that $\Delta_d=2\Delta_{\p}=2\Delta_{\bar{\p}}$.

In this section we shall prove the result below, which reduces the solving of the Poincar\'e-Lelong equation into the study of two parabolic equations, namely the heat equation and Hodge-Laplace heat equation, and obtaining relevant estimates.

\begin{thm}\label{thm-PL-general} Suppose $M$ has nonnegative Ricci curvature. Suppose further that the following are true.
\begin{itemize}
  \item[(a)] There is an $(1,1)$-form $\eta(x,t)$ satisfying
  \be\label{eq-heat-form-11}
\left\{
  \begin{array}{ll}
   & \eta_t(x, t)- \Delta \eta(x, t)   =0, \text{ in $M\times[0,\infty)$}; \\
    &\eta(x,0)  = \rho(x), \quad   x\in M,
  \end{array}
\right.
\ee
such that $\eta(x, t)$ is closed for all $t$, and for some $p>0$
\begin{equation}\label{eq-decay-form}
   \lim_{R\to\infty}\frac1{R^2V_o(R)}\int_0^T\int_{B_o(R)}||\eta||^p(x,t)d\mu(x) dt=0
\end{equation}
for all $T>0$. Moreover, $\lim_{t\to\infty}\eta(x,t)=0$.
  \item[(b)] There is a function $u(x)$ solving $\Delta u(x)=\tr(\rho)(x)$, and a solution $v(x, t)$ of
     \be\label{eq-heat-scalar}
\left\{
  \begin{array}{ll}
   & v_t(x, t)- \Delta v (x, t)   =0, \text{ in $M\times[0,\infty)$}; \\
   & v(x,0)  = 2 u(x), \quad   x\in M,
  \end{array}
\right.
\ee
such that for the same $p>0$,
\begin{equation}\label{eq-decay-scalar}
   \lim_{R\to\infty}\frac1{R^2V_o(R)}\int_0^T\int_{B_o(R)}|v|^p(x,t)+|u(x)|^pd\mu(x) dt=0
\end{equation}
for all $T>0$ and $\lim_{t\to\infty}\p\bar\p v(x,t)=0$.
\end{itemize}
Then $2\ii \p\bar\p u=\rho$.
\end{thm}
Before we prove the theorem, let us first recall the following consequence of K\"ahler identities.
\begin{lma}\label{lma-adjoint1} For   $(1,1)$-form $\eta$, we have
\be\label{eq-adjoint-1}
\p\bar\p \Lambda \eta=\ii\Delta_{\bar\p}\eta-\bar\p\Lambda\p\eta+\p\Lambda\bar\p\eta-\Lambda\p\bar\p\eta.
\ee
  Hence if $\Lambda \p \eta=\Lambda \bar{\p} \eta =\Lambda \p\bar{\p} \eta=0$, then
  \be\label{eq-adjoint}
  \p\bar\p \Lambda \eta=\ii\Delta_{\bar\p}\eta= \ii\Delta_{\p}\eta.
  \ee
   In particular, if $\eta$ is $d$-closed then \eqref{eq-adjoint} is true.
\end{lma}
\begin{proof} Recall the K\"ahler identities:
$
\p \Lambda-\Lambda\p=-\ii\bar\p^*, \ \ \bar\p \Lambda-\Lambda\bar\p= \ii \p^*.
$
 $\Delta_\p=\Delta_{\bar\p}$, we have that
\begin{equation*}
    \begin{split}
  \p\bar\p \Lambda \eta=&\p \Lambda\bar\p \eta+\ii\p\p^* \eta\\
  =&\ii\Delta_\p\eta+\p \Lambda\bar\p \eta-\ii\p^*\p\eta\\
  =&\ii\Delta_\p\eta+\p \Lambda\bar\p \eta-  \lf(\bar\p \Lambda-\Lambda\bar\p\ri)\p\eta.\\
    \end{split}
\end{equation*}
From this \eqref{eq-adjoint-1} follows.
\end{proof}

We also need the following maximum principle for solutions of the heat equation.
\begin{lma}\label{lma-Li-Tam} Suppose $(M^m,g)$ is a complete noncompact Riemannian manifold with nonnegative Ricci curvature and $u$ is a smooth nonnegative subsolution of the heat equation   on $M\times[0,T]$ such that there exists $R_i\to\infty$ and $p>0$ such that
\be\label{eq-subsolution}
\lim_{i\to\infty} \frac1{R_i^2V_o(R_i) }\int_0^T\int_{B_o(R_i)}u^p(y,s)d\mu(y) ds=0.
\ee
Then $u(x, t)\le \sup_{y\in M} u(y, 0)$.   In particular, if $u_1$ and $u_2$ are two solutions of the heat equation such that $|u_1|$ and $|u_2|$ satisfy the decay conditions \eqref{eq-subsolution} and if $u_1=u_2$ at $t=0$, then $u_1\equiv u_2$.
\end{lma}
\begin{proof}
By   \cite[Theorem  1.2]{LiTam1991}, for any $\e>0$,  there is a constant $C>0$ independent of $R$ such that
$$
\sup_{B_o(\frac12 R)\times[0,T]}u^p\le \frac{C}{R^2V_o(R)} \int_0^T\int_{B_o(R)}u^p(y,s)d\mu(y) ds+(1+\e)\lf(\sup_{y\in M} u(y, 0)\ri)^p
$$
if $R^2\ge 4T$. Let $R\to0$ and then let $\e\to0$,   the first assertion follows.

To prove the second assertion, apply the above argument to $\lf(|u_1-u_2|^2+\e\ri)^\frac12$ for $\e>0$ and let $\e\to0$.
\end{proof}
\begin{proof}[Proof of Theorem \ref{thm-PL-general}] Let $\eta$ be as in (a). Let $ \phi=\tr(\eta)$, namely $\Lambda(\eta)$. Then $\phi$  satisfies the heat equation in $M\times[0,\infty)$ with initial value $\tr(\rho)$. Let
$$
w(x,t)=-2\int_0^t \phi(x,s)ds.
$$
Then
$$
w_t-\Delta w=-2\tr(\rho),\ w(x,0)=0.
$$
Hence $\tilde v(x,t)=2u(x)-w(x,t)$ satisfies
$$
\tilde v_t-\Delta \tilde v=0,\ \tilde v(x,0)=2u(x).
$$
By \eqref{eq-decay-form}, \eqref{eq-decay-scalar} and Lemma \ref{lma-Li-Tam}, we conclude that $v=2u-w$.

On the other hand, by Lemma \ref{lma-adjoint1} and the fact that $\eta$ is closed, we have
\begin{equation*}
\begin{split}
\frac{d}{dt}\lf(\eta +\ii\p\bar\p w \ri)=&-2\Delta_{\bar\p}\eta+\ii\p\bar\p w_t\\
=&-2\Delta_{\bar\p}\eta-2\ii\p\bar\p \Lambda\eta\\
=&0.
\end{split}
\end{equation*}
At the same time, at $t=0$, $\eta +\ii\p\bar\p w(\cdot,t)=\rho$. Hence this equation holds for all $t$. That is to say,
$$
\eta+2\ii\p\bar\p u-\ii\p\bar v=\rho.
$$
Since $\lim_{t\to\infty}\eta(x,t)=\lim_{t\to\infty}\ii\p\bar\p v(x,t)=0$, we have $2\ii\p\bar\p u=\rho$.
\end{proof}
\begin{rem}\label{rem-general-1} From the proof it is easy to see that if in (a) we only assume that $\Lambda \p \eta=\Lambda \bar{\p} \eta =\Lambda \p\bar{\p} \eta=0$, then the conclusion of the theorem is still true. Moreover from the proof we see that $\eta$ is closed. In fact, one can check if $\eta$ satisfies the equation $\frac{\p}{\p t}\eta+\Delta_\p\eta=0$, then
$$
\Phi= \eta(t)+\int_0^t\lf(\bar\p^* \p+\p^*\bar\p\ri)\eta-\ii\int_0^t  \Lambda\p\bar\p\eta.
$$
is closed, provided that $\rho(x, 0)$ is closed. Hence in particular, $\Lambda \p \eta=\Lambda \bar{\p} \eta =\Lambda \p\bar{\p} \eta=0$ implies that $d\eta=0$.
\end{rem}

Hence in order to solve the Poincar\'e-Lelong equation, it is sufficient to find $\eta$ and $u, v$ as in the theorem.

By the previous work \cite{NT-2003, NST}, under certain average growth condition on $\|\rho\|$ we can find $u$ and $v$ satisfying (b) of Theorem \ref{thm-PL-general}. We now list this known result and a useful estimate below.

 Let $(M,g)$ be a complete noncompact Riemannian manifold and let $o\in M$ be a fixed point. For a smooth function  $f$ on $M$, let
\be\label{eq-k(f)}
k_f(r)=\frac1{V_o(r)}\int_{B_o(r)}|f|.
\ee
where $B_x(r)$ is the geodesic ball of radius $r$ with center at $x$ and $V_x(r)$ is the volume of $B_x(r)$. First recall the following result from  \cite[Lemma 1.1]{NT-2003}:
\begin{lma}\label{lma-Ni-Tam} Let $(M,g)$ be a complete noncompact Riemannian manifold with nonnegative Ricci curvature and let $o\in M$ be a fixed point. Let $h\ge0$ be a continuous function. Let $H(x,y,t)$ be the heat kernel and let
$$
v(x,t)=\int_M H(x,y,t)h(y)d\mu(y).
$$
Assume that $v$ is defined on $M\times[0,T]$ with
\be\label{eq-NT}
\liminf_{r\to\infty}\exp\lf(-\frac{r^2}{20T}\ri)\int_{B_o(r)}h=0.
\ee
Then for any $r^2\ge t>0$ and $p\ge 1$
$$
\frac1{V_o(r)}\int_{B_o(r)}v^p(x,t) dx\le C(n,p)\lf[k_{h^p}(4r)+t^{-p}\lf(\int_{4r}^\infty s\exp\lf(-\frac{s^2}{20t}\ri)k_h(s)ds\ri)^p\ri].
$$
\end{lma}
\begin{proof} Note  the proof in \cite{NT-2003} (page 467-468) can be carried over because    only  \eqref{eq-NT} is needed for the integration by parts.
\end{proof}

\begin{prop}\label{prop-solving u v} Let $(M,g)$ be a complete noncompact Riemannian manifold with nonnegative Ricci curvature and let $o\in M$ be a fixed point.   Let $f$ be a smooth function on $M$ such that
$$
\int_0^\infty k_f(r) dr<\infty.
$$
Then we can find functions $u$ and $v$ with $\Delta u=f$, $v$ satisfying \eqref{eq-heat-scalar}  such that \eqref{eq-decay-scalar} is true for $p=1$, and $\lim_{t\to\infty}\p\bar\p v(x,t)=0$. Moreover $u$ satisfies \eqref{bound}. In fact, $u$ and $v$ are  given by
\begin{eqnarray*}
u(x)&=&\int_M \lf(G(o,y)-G(x,y)\ri) f(y) d\mu(y), \text{\ if $M$ is nonparabolic}\\
v(x,t)&=&\int_M H(x,y,t) u(y) d\mu(y).
\end{eqnarray*}
Here $G(x, y)$ is the minimal positive Green's function of $M$ (if $M$ is nonparabolic) and $H(x, y, t)$ is the heat kernel of $M$.
\end{prop}
\begin{proof} First consider the case that $M$ is nonparabolic. The existence of $u$ and $v$ given by the expressions in the proposition so that $\lim_{t\to\infty}\p\bar\p v(x,t)=0$ follows from \cite[Lemma 6.1]{NT-2003}. By the proof of \cite[Lemma 6.1]{NT-2003}, we have
$$
k_u(r)=o(r^2).
$$
From this and Lemma \ref{lma-Ni-Tam}, \eqref{eq-decay-scalar} is true for $p=1$ in this case.

In general, by considering $\widetilde M=M\times \R^4$, we can find $\widetilde u$ as above. By \cite[p.458--460]{Fan2007}, $\widetilde u$ is independent of $y\in \R^4$ for $(x,y)\in M\times \R^4$. Let $u(x)=\widetilde u(x,y)$. Then $\Delta u=f$ and $k_u(r)=o(r^2)$. The existence of $v$ is as in the previous case.
\end{proof}

To apply Theorem \ref{thm-PL-general} we also need a {\it $d$-closed} solution of the Hodge-Laplace heat equation with estimate (\ref{eq-decay-form}). The existence of such solutions is the content of  the next two sections.

\section{Solution of the Hodge-Laplace heat equation: preliminary results}
In this section we  collect some basic results needed for the later discussions and construct a solution  of the Hodge-Laplace heat equation (\ref{eq-heat-form-11}) via suitable approximation.

First  we need  the existence of an exhaustion function on complete manifolds with nonnegative Ricci curvature with the property that  it has  bounded complex Hessian if additionally  $(M^n, g)$ is a K\"ahler manifold with nonnegative quadratic orthogonal bisectional curvature.

Recall that a K\"ahler manifold $(M^n,g)$  is said to have {\it nonnegative quadratic orthogonal bisectional curvature} if at any point and any unitary frame $\{e_i\}$,
\be\label{assumption-qu-or}
\sum_{i,j}R_{i\bar i j\bar j}(a_i-a_j)^2\ge0
\ee
for all real numbers $a_i$. Let $o\in M$ be a fixed point, $r(o, x)$ be the distance function to $o$.

Let
$$
\zeta(x,t)=\int_MH(x,y,t)r(o, y)d\mu(y)
$$
be the positive solution of the heat equation with $r(o, x)$ being the initial value. Here $H(x,y,t)$ is the heat kernel.
\begin{prop}\label{prop-pre} (i) Assume that $(M, g)$ has nonnegative Ricci curvature, then for any $t>0$ $\zeta(x, t)$ is a smooth exhaustion function. Moreover $|\nabla \zeta|(x, t)\le 1$.

(ii) Assume additionally that $(M, g)$ has nonnegative quadratic orthogonal bisectional curvature, then for any $t>0$, there exists $C$ such that $\|\p \bar{\p} \zeta\|(x, t)\le C$. Furthermore $C$ can be chosen with $C=\frac{1}{\sqrt{2t}}$. \end{prop}
\begin{proof} The part (i) follows from Corollary 1.1 and Lemma 1.4 of \cite{NT-2003}. For part (ii), note that the Bochner formula implies that
\begin{equation}\label{pre-1}
\heat |\nabla \zeta|^2\ge 2\|\nabla^2\zeta\|^2.
\end{equation}
As in the proof of Lemma 3.1 of \cite{NT-2003}, by multiplying a suitable cut-off function to (\ref{pre-1}) and performing the integration by parts we have that for any $x\in M$,  and for any $T, R>0$,
\begin{equation} \label{pre-2}
\int_0^T\aint_{B_x(R)}\|\nabla^2\zeta\|^2 \le C\lf(R^{-2}\int_0^T\aint_{B_x(2R)}|\nabla\zeta|^2+ \aint_{B_x(2R)}|\nabla r|^2\ri)\le C(1+R^{-2}T)
\end{equation}
for some constant $C$ independent of $x$ and $R$. Now one can apply Theorem 1.1 of \cite{LiTam1991} to conclude the bound of $\|\p\bar{\p} \zeta\|$, since by the proof of Lemma 2.1 of \cite{NT-2003} (see also \cite{GK}, Lemma 5.1 of \cite{Wi}) one only needs the nonnegativity of the  quadratic orthogonal bisectional curvature to conclude that $\|\p\bar{\p} \zeta\|(x, t)$ is a sub-solution to the heat equation (see also (\ref{eq-11-form special}) below). More precisely under the curvature assumption of (ii) we have that
\begin{equation}\label{pre-3}
\heat \|\p\bar{\p} \zeta\|^2(x, t) \ge 2 \|\nabla \p\bar{\p} \zeta\|^2(x, t).
\end{equation}
Putting (\ref{pre-2}) and (\ref{pre-3}) together we have that
\begin{equation}\label{pre-4}
\heat \left(|\nabla \zeta|^2+2t\|\p\bar{\p} \zeta\|^2\right)\ge 0.
\end{equation}
The estimate $|\nabla \zeta|(x, t)\le 1$ and (\ref{pre-3}) enable one to apply the maximum principle to
$|\nabla \zeta|^2+2t\|\p\bar{\p} \zeta\|^2$, and conclude that
$$
|\nabla \zeta|^2(x, t)+2t\|\p\bar{\p} \zeta\|^2(x, t) \le 1.
$$
This gives the precise upper bound claimed in part (ii).
\end{proof}

The Bochner formulae we have applied above are the special cases of the general formulae described below. Let $\eta$ be a $(p, q)$-form valued in a holomorphic Hermitian vector bundle $E$ with local frame $\{E_\a\}$ and locally $\eta=\sum \eta^\a E_{\a}$.
$$
\eta^\a = \frac{1}{p!q!}\sum \eta^\a_{I_p\bar{J}_q}dz^{I_p}\wedge dz^{\bar{J}_q}.
$$
Here $I_p=(i_1,\cdots,i_p), \bar{J}_q=(\bar{j}_1, \cdots,\bar{j}_q)$ and $dz^{I_p}=dz^{i_1}\wedge
\cdots\wedge dz^{i_p}$, $dz^{\bar{J}_q}=dz^{\overline{j_1}}\wedge\cdots \wedge dz^{\overline{j_q}}.$
The Kodaira-Bochner formulae include the following two identities:
\begin{eqnarray}
 (\Delta_{\bar{\partial}}\eta)^{\a}_{I_p \bar{J}_q} & = & -\sum_{ij}g^{\bar{j}i}\nabla_i\nabla_{\bar{j}}\eta^{\a}_{I_p \bar{J}_q}+
\sum_{\nu = 1}^q \Omega^{\a\bar{l}}_{\b\, \bar{j}_\nu}\eta^{\b}_{I_p \bar{j}_1\cdots (\bar{l})_\nu \cdots\bar{j}_q}
 \nonumber \\
& \,&  +\sum_{\nu =1}^{q}R^{\bar{l}}_{\,\,
\bar{j}_\nu}\eta^{\a}_{I_p \bar{j}_1\cdots(\bar{l})\cdots\bar{j}_q}
-\sum_{\mu = 1}^{p}\sum_{\nu =1 }^{q}R^{\,
k\bar{l}}_{i_\mu\,\,\bar{j}_\nu}\eta^{\a}_{i_1\cdots(k)_\mu\cdots
 i_p \bar{j}_1\cdots(\bar{l})_\nu\cdots\bar{j}_q}\label{eq:43}
\end{eqnarray}
and
\begin{eqnarray}
 (\Delta_{\bar{\partial}}\eta)^{\a}_{I_p\bar{J}_q} & =&
-\sum_{ij}g^{i\bar{j}}\nabla_{\bar{j}}\nabla_{i}
\eta^\a_{I_p\bar{J}_q} + \sum_{\nu = 1}^q \Omega^{\a\bar{l}}_{\b
\bar{j}_\nu} \eta^{\b}_{I_p \bar{j}_1\cdots (\bar{l})_\nu
\cdots\bar{j}_q} -\sum_{\b}\Omega_{\b}^{\a}
\eta^{\b}_{I_p \bar{J}_q} \nonumber \\
& \, & +\sum_{\mu =1 }^{p}R_{i_\mu}^{\,\,
k}\eta^\a_{i_1\cdots(k)_\mu\cdots i_p \bar{J}_q} -\sum_{\mu =
1}^{p}\sum_{\nu =1 }^{q}R^{\, \, k\bar{l}}_{i_\mu\, \, \,
\bar{j}_\nu}\eta^{\a}_{i_1\cdots(k)_\mu\cdots
 i_p \bar{j}_1\cdots(\bar{l})_\nu\cdots\bar{j}_q}, \label{eq:44}
\end{eqnarray}
where $(k)_{\mu}$ means that the index in the $\mu$-th position is replaced by $k$. Here
$$
\Theta^\a_{\beta}=\frac{\sqrt{-1}}{2\pi}\sum \Omega^\a_{\beta \, i\bar{j}}dz^i\wedge d\bar{z}^j
$$
is the curvature of $E$ and $\Omega^\a_\beta=\Lambda \Theta^\a_{\beta}$ is the mean curvature.

\begin{lma}\label{lma-pre-1} Let $(M,g)$ be a \K manifold.

\begin{enumerate}
  \item Suppose that $M$ has nonnegative quadratic bisectional curvature. Let $\eta(x, t)$ be a $(1,1)$-form satisfying
$$
\frac{\p}{\p t} \eta(x, t)-\Delta \eta(x, t)=\xi(x, t)
$$
with $\xi(x, t)$ being another $(1,1)$-form. Then
\begin{equation}\label{eq-11-form special}
\heatt \|\eta\|^2(x, t) \le -2\|\nabla \eta\|^2(x, t) +2\langle \xi, \eta\rangle (x, t).
\end{equation}
In particular, for any $\e>0$,
\bee
\heatt \lf(\|\eta\|^2(x, t)+\e\ri)^\frac12 \le     ||\xi||(x,t).
\eee
  \item Suppose that $M$ has nonnegative Ricci curvature. Let $\sigma(x, t)$ be a $(1,0)$-form satisfying
$$
\frac{\p}{\p t}\sigma(x, t)-\Delta \sigma(x, t)=0.
$$
 Then
\begin{equation}\label{eq-11-form special-2}
\heatt \|\sigma\|^2(x, t) \le -2\|\nabla \sigma\|^2(x, t).
\end{equation}
In particular, for any $\e>0$,
\bee
\heatt \lf(\|\sigma\|^2(x, t)+\e\ri)^\frac12 \le     0.
\eee
\end{enumerate}

\end{lma}
\begin{proof}
  Applying (\ref{eq:43}) and (\ref{eq:44}) to   $\eta(x, t), \sigma(x, t)$ and using the curvature assumption, the results follow.
\end{proof}

Our strategy of constructing  a global solution $\eta$ to the Hodge-Laplace heat equation is by two  approximations. First we construct $\rho^{(i)}$ by multiplying the initial $d$-closed $(1,1)$-form $\rho$ by suitable cut-off functions $\phi^{(i)}$,  which we describe below.

Let $\kappa(s)$ be a smooth cut-off function on $\R$ such that $\kappa(s)=0$ for $|s|>1$. For  a sequence $\{R_i\}$ with $R_i\to \infty$, let
$$\phi^{(i)}(x)=\kappa\left(\frac{\zeta(x, 1)}{R_i}\right)$$
where $\zeta(z, t)$ is the exhaustion function constructed by solving the heat equation with distance function as the initial data. Proposition \ref{prop-pre} implies that there exists an absolute constant $C_1>0$ such that
\begin{equation}\label{eq-pre-1}
|\nabla \phi^{(i)}|(x) +\|\p\bar{\p} \phi^{(i)}\|\le \frac{C}{R_i}.
\end{equation}

Let $\rho^{(i)}(x)=\phi^{(i)}(x) \rho(x)$. In general one can not expect that $\rho^{(i)}$ is closed. But we have the following estimate:
\begin{equation} \label{eq-pre-2}
\|\p \rho^{(i)}\|(x)+\|\bar{\p}\rho^{(i)}\|(x)+\|\p\bar{\p} \rho^{(i)}\|(x)\le \frac{C_1}{R_i}\|\rho\|(x).
\end{equation}

A solution $\eta^{(i)}(x, t)$ of the Hodge-Laplace heat equation  with initial data $\rho^{(i)}(x)$ is constructed via the  compact exhaustion detailed as follows. Let $\{\Omega_\nu\}$ be a sequence of exhaustion domains. Let $\Phi_{\nu}$ be the solution of the following initial boundary value problem:
\be\label{eq-heat-form-2}
\left\{
  \begin{array}{ll}
&    \Phi_t- \Delta \Phi   =0, \text{ in $\Omega_\nu \times(0,\infty)$}; \\
  &  \lim_{t\to 0}\Phi(x,t)  = \rho^{(i)}(x), \ x\in\Omega_\nu;\\
   & \mathbf{n}\Phi  = 0,  \text{ on $\p\Omega_\nu \times(0,\infty)$};\\
    &\mathbf{t}\Phi=0, \text{ on $\p\Omega_\nu \times(0,\infty)$}.
  \end{array}
\right.
\ee
Here $\mathbf{n} \Phi$ and $\mathbf{t} \Phi$ denote the normal and tangential parts of $\Phi$ (see \cite{Morrey} for the basic definitions and related properties). The solvability  of (\ref{eq-heat-form-2}) is classical. See for example \cite{Evans}, \cite{Morrey}, \cite{LSU}. For $\nu$ large, the solution $\Phi_\nu$ is smooth since the initial and boundary values are compatible.  Let
\begin{equation}\label{def}
v^{(i)}(x, t)\doteqdot\int_M H(x, y, t) \|\rho^{(i)}\|(y)\, d\mu(y) \mbox{ and }  v(x, t)\doteqdot\int_M H(x, y, t)\|\rho\|(y)\, d\mu(y).\end{equation}
 Clearly $v^{(i)}(x, t)\le v(x, t)$, provided $v$ is defined. The Bochner formula (\ref{eq-11-form special}) and the maximum principle imply that
 $$
 \|\Phi_\nu\|(x, t)\le v^{(i)}(x, t).
 $$
Thus by the Schauder's estimate, after possibly passing to a subsequence, $\{\Phi_\nu\}$ converges to a limit solution $\eta^{(i)}(x, t)$, which is a $(1,1)$-form with the initial value $\rho^{(i)}(x)$. Note that if $\rho$ is real, then $\eta^{(i)}$ is real by uniqueness. To summarize, we have:

\begin{lma}\label{lma-pre-2} Let $\eta^{(i)}$ be as above. Suppose that $v(x, t)$ in (\ref{def})  is well-defined. Then, after possibly passing to a subsequence, $\{\eta^{(i)}\}$  converges to a (1,1)-form $\eta$ satisfying the Hodge-Laplacian heat equation on $M\times[0,\infty)$ with initial value $\rho$. Moreover,
\begin{equation}\label{eq-pre3}
\|\eta^{(i)}\|(x, t)\le v^{(i)}(x, t) \mbox{ and } \|\eta\|(x, t)\le v(x, t).
\end{equation}
In particular for each $i$, $\|\eta^{(i)}\|$ is uniformly bounded. 
\end{lma}

Next we will prove that
under certain conditions,  $\eta$ satisfies the conditions in Theorem 2.1(a).

\section{Global solutions of the Hodge-Laplace heat equation }

 In this section we prove the following result on the global solutions of the Hodge-Laplace heat equation.

\begin{thm}\label{thm-heat-form-2}
Let $(M, g)$ be a complete K\"ahler manifold with nonnegative quadratic orthogonal bisectional curvature and with nonnegative Ricci curvature. Assume that $\rho$ is a $d$-closed $(1,1)$-form such that $f=\|\rho\|$ satisfies
\be\label{f-decay}
\limsup_{R\to \infty} \frac{k_f(R)}{R^2}=0.
\ee
Then there exists a solution of
\be\label{eq-heat-form-3}
\left\{
  \begin{array}{ll}
    \eta_t- \Delta \eta   &=0, \text{ in $M\times[0,\infty)$}; \\
    \eta(x,0)  &= \rho(x), \quad   x\in M.
  \end{array}
\right.
\ee
such that $\eta$ is a closed $(1,1)$-form.   Furthermore we have that
 \begin{itemize}
   \item [(a)]  for any $T>0$,
\be\label{eq-eta-decay}
\lim_{R\to\infty}\frac1{R^2V_o(R)}\int_0^T\int_{B_o(R)}||\eta||(x,t) d\mu(x)dt=0;
\ee
   \item [(b)]$\lim_{t\to\infty}\eta(x,t)=0$ for all $x\in M$ provided that $\lim_{R\to \infty} k_f(R)\to 0$.
 \end{itemize}
\end{thm}

\begin{proof} By the decay assumption on $||\rho||$,
$$
v(x,t)=\int_MH(x,y,t)||\rho||(y)d\mu(y)
$$
is well-defined by \cite[Lemma 2.1]{NT-2003}. Let $\{\eta^{(i)}(x, t)\}$ and $\eta$ be the solutions of the Hodge-Laplace heat equation obtained in Lemma \ref{lma-pre-2}, which are real (1,1)-form.
Then $\|\eta\|(x, t)\le v(x, t)$.
Now by \eqref{f-decay} we have that $\frac{k_f(4R)}{R^2}\to 0$
as $R\to\infty$. By Lemma \ref{lma-Ni-Tam} and the fact that $\|\eta\|(x,t)\le v(x,t)$, we conclude that \eqref{eq-eta-decay} is true for any $T>0$.

By \cite[Theorem 1.1]{LiTam1991}, if additionally assume that $\lim_{R\to \infty} k_f(R)= 0$,  we have  that $\lim_{t\to \infty}v(x,t)=0$. Hence $\lim_{t\to\infty} \eta(x,t)=0 $.

It remains to prove $d\eta=0$. By Remark \ref{rem-general-1}, it is sufficient to prove that $\Lambda\p\eta=\Lambda\bar\p\eta=\Lambda\p\bar\p\eta=0$. These will be established via  lemmas below. With them we complete the proof of the theorem.
\end{proof}

\begin{lma}\label{lma-global-1} Let $\eta^{(i)}$ and $\eta$ as in the proof of Theorem \ref{thm-heat-form-2}. Let $\sigma_1^{(i)}=\Lambda\p\eta$ and $\sigma_2^{(i)}=\Lambda\p\eta^{(i)}$. Then
\be\label{eq-global-1}
||\sigma_j^{(i)}||(x,t)\le  C\int_MH(x,y,t)||\rho_i||(y)d\mu(y)
\ee
for $j=1,2$. In particular, for each $i$, $||\sigma_j^{(i)}||$ are uniformly bounded in $M\times[0,\infty)$. Moreover, $\Lambda\p\eta=\Lambda\bar\p\eta=0$.
\end{lma}
\begin{proof} By Lemma \ref{lma-pre-2}, for each $i$, $||\eta^{(i)}||$ is uniformly bounded on $M\times[0,\infty)$. By Lemma \ref{lma-pre-1}
\bee
\heatt ||\eta^{(i)}||^2\le -2||\nabla \eta^{(i)}||^2.
\eee
As in the proof of Proposition \ref{prop-pre}, for all there is $C>0$, such that for $T>0$, $r>0$,
\bee
\int_0^T\aint_{B_o(r)}||\nabla \eta^{(i)}||^2\le C(1+r^{-2}T).
\eee
Hence
\bee
\int_0^T\aint_{B_o(r)}||\sigma_1^{(i)}||^2\le C(1+r^{-2}T).
\eee
On the other hand, by   Lemma \ref{lma-pre-1} again, for any $\e>0$,
\bee
\heatt \lf(||\sigma_1^{(i)}||^2+\e\ri)^\frac12\le 0.
\eee
Hence one can apply the maximum principle to conclude that
\bee
\lf(||\sigma_1^{(i)}||^2(x,t)+\e\ri)^\frac12\le \int_MH(x,y,t)\lf(||\Lambda\p\rho^{(i)}||(y)+\e\ri)d\mu(y).
\eee
Let $\e\to0$, we have
\bee
||\sigma_1^{(i)}||(x,t) \int_MH(x,y,t) ||\Lambda\p\rho^{(i)}||(y)d\mu(y).
\eee
Similarly, one can prove
\bee
||\sigma_2^{(i)}||(x,t) \int_MH(x,y,t) ||\Lambda\bar\p\rho^{(i)}||(y)d\mu(y).
\eee
The last assertion follows from \eqref{eq-pre-2} and the fact that a subsequence of $\eta^{(i)}$ converge locally uniformly to $\eta$.
\end{proof}

Next we want   to prove that $\Lambda \p\bar{\p} \eta(x, t)=0$. For that let
$$\theta=\int_0^t \Lambda \p\bar{\p} \eta(x, \tau)\, d\tau, \quad  \mbox{ and correspondingly  }\quad  \theta^{(i)}(x, t)=\int_0^t  \Lambda \p\bar{\p} \eta^{(i)}(x, \tau)\, d\tau.$$ For that  we need more lemmas.

\begin{lma} \label{help-lma1}  For any $T>0$, there exists $C_i>0$  such that
\begin{equation}\label{eq-theta-1}
\int_0^T\aint_{B_o(R)} \|\theta^{(i)}\|^2(x, t)\, d\mu(x)\, dt \le C_i
\end{equation}
for all $R$.
\end{lma}
\begin{proof} Let $w^{(i)}(x, t)=\Lambda \eta^{(i)}(x, t)$. By Lemma \ref{lma-pre-2},   $w^{(i)}(x, t)$ is uniformly bounded. It also satisfies the heat equation.
Using the differential equation/inequality
\begin{eqnarray*}
\heat \left(w^{(i)}\right)^2(x, t)&=&2|\nabla w^{(i)}|^2(x, t),\\
\heat |\nabla w^{(i)}|^2(x, t) &\ge& 2\|\nabla^2 w^{(i)}\|^2(x, t)
\end{eqnarray*}
integration by parts as in the proof of Proposition \ref{prop-pre}, for $R^2\ge T$
\begin{equation}\label{eq:57}
\int_0^T \aint_{B_o(R)}\left(|\nabla w^{(i)}|^2(x, t)+\|\nabla^2 w^{(i)}\|^2(x, t)\right)\, d\mu(x)\, dt\le C_i.
\end{equation}
for some $C_i>0$ independent of $R$.
Let $\sigma^{(i)}_1(x, t)=\Lambda \bar{\p} \eta^{(i)}$  and $\sigma_2^{(i)}(x, t)=\Lambda \p \eta^{(i)}(x, t)$.
By Lemma \ref{lma-adjoint1},
\begin{equation}\label{eq:58}\theta^{(i)}(x, t)=
\int_0^t \p\bar{\p}w^{(i)}(x, \tau)d\tau +\int_0^t\sqrt{-1}\Delta_{\bar{\p}} \eta^{(i)}(x, \tau) +\p \sigma^{(i)}_1(x, \tau) -\bar{\p} \sigma_2^{(i)}(x, \tau) d\tau.
\end{equation}
Note 
\begin{equation}\label{eq:55}
\|\int_0^t\sqrt{-1}\Delta_{\bar{\p}} \eta^{(i)}(x, \tau)\, d\tau \|=\|\int_0^t\ii \eta^{(i)}_\tau d\tau\|\le v^{(i)}(x, t)+\|\rho^{(i)}\|(x).
\end{equation}
On the other hand  by Lemma \ref{lma-pre-1}
\begin{equation}\label{eq:51}
\heat \|\sigma^{(i)}_j\|^2(x, t)\ge 2\|\nabla \sigma^{(i)}_j\|^2(x, t), \mbox{ for } j=1, 2,
\end{equation}
and  $\|\sigma^{(i)}_j\|^2$ are uniformly bounded by Lemma \ref{lma-global-1}.
 Integrating by parts after multiplying (\ref{eq:51}) by a cut-off function as before we have
\begin{equation}\label{eq:56}
\int_0^T \aint_{B_o(R)} \|\nabla \sigma^{(i)}_1\|^2(x, t) +\|\nabla \sigma^{(i)}_2\|^2(x, t)\, d\mu(x)\, dt \le C_i.
\end{equation}
The result follows from (\ref{eq:57}), (\ref{eq:58}), (\ref{eq:55})and (\ref{eq:56}).
\end{proof}

 Lemma \ref{help-lma1} allows us to apply the maximum principle to $\|\theta^{(i)}\|(x, t)$ once we can establish that $\|\theta^{(i)}\|(x, t)$ is a subsolution of a heat type equation. To this end we have the following lemma.

\begin{lma}\label{lma-help2} For any $\e>0$,
\be
\heat\lf(||\theta^{(i)}||(x, t)+\e\ri)^\frac12\ge  -||\Lambda\p\bar\p\rho^{(i)}||(x).
\ee
\end{lma}
\begin{proof} First direct calculation shows that
\begin{eqnarray*}
\frac{\p}{\p t}\theta^{(i)}(x, t)-\Delta \theta^{(i)}(x, t)&=&\Lambda\p\bar\p\eta^{(i)}(x, t)-\int_0^t\Delta \Lambda\p\bar\p\eta^{(i)}(x, \tau)\, d\tau\\
&=&\Lambda\p\bar\p\eta^{(i)}(x, t)-\int_0^t\Lambda\p\bar\p\Delta\eta^{(i)}(x, \tau)\, d\tau\\
&=&\Lambda\p\bar\p\eta^{(i)}(x, t)-\int_0^t\Lambda\p\bar\p (\eta^{(i)})_t(x, \tau)\, d\tau\\
&=&\Lambda\p\bar\p\rho^{(i)}(x).
\end{eqnarray*}
The claimed inequality follows from (\ref{eq-11-form special}) and Lemma \ref{lma-pre-1}.
\end{proof}

\begin{lma}\label{lma-global-2} $\theta(x,t)=0$ for all $x, t$.
\end{lma}
 \begin{proof} Let
 $$
 z^{(i)}(x, t)=\int_0^t \int_M H(x, y, \tau) ||\Lambda\p\bar\p\rho^{(i)}||(y)\, d\mu(y)\, d\tau.
 $$
It is easy to see that
$$
\heat z^{(i)}(x, t)= -||\Lambda\p\bar\p\rho^{(i)}||(x).
$$
Note that $z^{(i)}$ is bounded. By Lemmas \ref{help-lma1}, \ref{lma-help2}  and the maximum principle Lemma \ref{lma-Li-Tam}, we conclude that
\begin{eqnarray*}
\|\theta^{(i)}\|(x, t) &\le& z^{(i)}(x, t)\\
&\le& \frac{C_1}{R_i}\int_0^t \int_M H(x, y, \tau) \|\rho\|(y)\, d\mu(y)\, d\tau\\
 &\le& \frac{C_1}{R_i}\int_0^t v(x, \tau)\, d\tau\\
 &\to & 0
\end{eqnarray*}
as $i\to \infty$. Here we have used \eqref{eq-pre-2}.
 \end{proof}

 By the definition of $\theta$, we conclude that  $\Lambda \p\bar{\p}\eta =0$. The proof of Theorem \ref{thm-heat-form-2} is completed.

\section{An alternate proof of the gap theorem}

In this section we shall give an alternate proof of the gap theorem of \cite{PL-remove}. The proof here avoids the use of the {\it relative monotonicity}. It uses a Li-Yau-Hamilton type estimate in \cite{NN} together with Lemma \ref{lma-global-1}. The proof also avoids to solve the Poincar\'e-Lelong equation.

\begin{thm}\label{gap1}  Let $(M, g)$ be a complete K\"ahler manifold with nonnegative bisectional curvature. Assume that $\rho\ge 0$ is a $d$-closed $(1,1)$-form.
Suppose that
\begin{equation}\label{ave-decay1}
\int_0^r s\aint_{B_{o}(s)}\|\rho\|(y)\, d\mu(y)\, ds =o(\log r)
\end{equation}
for some $o\in M$. Then $\rho\equiv 0$.
\end{thm}
\begin{proof} By the assumption \eqref{ave-decay1}, we can apply Theorem \ref{thm-heat-form-2}. Let $\eta^{(i)}$ and $\eta$ as in the proof of Theorem \ref{thm-heat-form-2} with $||\eta^{(i)}||$ is bounded and

$$\|\eta\|(x, t)\le \int_M H(x, y, t) \|\rho\|(x)\, d\mu.
$$
Since $\rho_i$ is nonnegative, by \cite[Theorem 2.1]{NT-2003}, we can conclude that $\eta^{(i)}$ is nonnegative. Hence $\eta$ is nonnegative. By \cite[Theorem 1.1]{NN}, we have
\be\label{lyh-help1}
w_t+\frac12\lf(\bar\p^*\Lambda\bar\p + \p^*\Lambda \p\ri)\eta+\frac{w}t\ge0
\ee
where $w=\Lambda\eta$. By Lemma \ref{lma-global-1}, $\Lambda\bar\p\eta=\Lambda \p \eta=0$. Let $u(x)$ be the solution of $\Delta u=\Lambda\rho$ obtained in \cite{NST}, which satisfies the estimates \eqref{bound}. Let $v(x,t)$ be the solution of the heat equation with initial data $2u$. By the Proposition \ref{prop-solving u v} and the estimates of $u$ and $\eta$, we conclude that
$$
v(x,t)=2u(x)+\int_0^t w(x,\tau)d\tau.
$$
By Proposition \ref{prop-solving u v}, through the previous work \cite{NST},  for any $x_0\in M$ we have that
$$u(x)=o(\log (r(x_0, x)))$$
as $x\to \infty$, where $r(x_0, x)$ denotes the distance function as before.
By the moment type estimate (cf. Theorem 3.1 of \cite{N02})
we have that $v(x_0, t)=o(\log t)$ as $t\to \infty$. This fact, together with (\ref{lyh-help1}) implies that $\Lambda \rho(x_0, t)=0$ for any $t>0$. Otherwise assume that $\Lambda \rho(x_0, t_0)=\delta>0$ for some $t_0> 0$. Note that we have $v_t= 2\Lambda \eta =2w$. Then by (\ref{lyh-help1}) we have that  for all $t\ge t_0$,
$$
v_t(x, t)\ge \frac{2t_0\delta}{t},
$$
 which in particular implies that
 $$
 v(x_0, t)\ge 2t_0\delta \log \left(\frac{t}{t_0}\right)-v(x_0, t_0).
 $$
This contradicts  $v(x, t)=o(\log (t))$. The contradiction implies that $\Lambda \eta(x_0, t)=0$ for any $t\ge 0$. Hence we have that $\Lambda  \rho(x)\equiv 0$ which implies the claim $\rho\equiv 0$ by the nonnegativity of $\rho$.

\end{proof}

\section{Solution of Poincar\'e-Lelong equation}\label{s-initial boundary}
In this section we shall prove the result generalizing \cite[Theorem 6.1]{NT-2003}.
With the notations of previous sections we can state the main theorem.

\begin{thm}\label{thm-PL equation}
Let $(M^n,g)$ be a complete noncompact \K manifold (of complex dimension $n$) with nonnegative Ricci curvature and nonnegative quadratic orthogonal bisectional curvature.    Suppose that $\rho$ is a smooth $d$-closed   real $(1,1)$-form on $M$ and let $f=||\rho||$ be the norm of $\rho$.  Suppose that
\be\label{eq-rho decay}
\int_0^\infty k_f(r)dr<\infty.
\ee
Then there is a smooth function $u$ so that $\rho=\ii \p\bar\p u$. Moreover, for any $0<\e<1$, the estimate (\ref{bound}) holds.
\end{thm}
\begin{proof} Observe that the volume comparison implies  $k_f(r)\le 2^{2n} k_f(s)$ for all $s\in (r, 2r)$. Thus (\ref{eq-rho decay}) implies that $\lim_{r\to \infty} k_f(r)\, r = 0$.  Thus Theorem \ref{thm-PL equation} follows  from Theorems \ref{thm-PL-general}, \ref{thm-heat-form-2} and Proposition \ref{prop-solving u v}.
\end{proof}

 Finally we make some comments on the relation of conditions of {\it nonnegative bisectional curvature} (NB), {\it nonnegative orthogonal bisectional curvature} (NOB), {\it nonnegative quadratic orthogonal bisectional curvature} (NQOB) (\ref{assumption-qu-or}), and the following curvature condition on $(2,1)$-forms $\sigma$:
   \begin{equation}\label{eq-curvature assumption}
-\langle E(\sigma), \sigma\rangle \ge -a^2 |\sigma|^2
\end{equation}
 for some $a>0$, where
\begin{equation}\label{eq-BL-formula-2}
  \lf(E(\sigma)\ri)_{ij\bar k} =- R_i^l \sigma_{lj\bar k}-R_j^l\sigma_{il\bar k}-R_{\bar k}^{\bar l}\sigma_{ij\bar l}+2R_{i\ \ \bar k}^{\  l\bar m\ }  \sigma_{lj\bar m}+2R_{j\ \ \bar k}^{\ l\bar m\ }  \sigma_{il\bar m}.
\end{equation}
  For 
   condition (\ref{eq-curvature assumption}) with $a=0$,  we abbreviate  as (NCF) (nonnegativity associated with certain $3$-forms), as well as an invariant representation of them. This condition on $(2,1)$-forms arises from the Bochner formula on $(2,1)$-forms, which is related to a previous method of constructing global $d$-closed solutions to Hodge-Laplace heat equation.  Algebraically  (NB) is stronger than (NOB), which in turn is stronger than (NQOB). We introduce some notations for  convenience. First the curvature operator of a K\"ahler manifold can be viewed as bilinear form on $\mathfrak{gl}(n, \mathbb{C})$ (which can be identified with $\wedge^{1, 1}(\mathbb{C}^n)$ via the metric) in the sense that for any $X\wedge \overline{Y}$ and  $Z\wedge \overline{W}$
$$
\langle \operatorname{Rm}(X\wedge\overline{Y}),  \overline{Z\wedge \overline{W}}\rangle \doteqdot
\operatorname{Rm}(X\wedge\overline{Y} , \overline{Z\wedge \overline{W}})= R_{X\overline{Y}W \overline{Z}}.
$$
Hence for any $\Omega=(\Omega^{i\bar{j}})$, it is easy to check that $\langle \operatorname{Rm}(\Omega), \overline{\Omega}\rangle =R_{i\bar{j}k\bar{l}}\Omega^{i\bar{j}}\overline{\Omega^{k\bar{l}}} \in \mathbb{R}$. Hence one can identify (cf. \cite{Wi}) the condition  (NB) as
\be\label{NB}
\{\operatorname{Rm}\, |\, \langle \operatorname{Rm} (\Omega), \overline{\Omega}\rangle\ge 0, \text{ for any } \Omega, \operatorname{rank}(\Omega)=1\}.
\ee
Similarly, condition (NOB) is equivalent to
\be\label{NOB}
\{\operatorname{Rm}\, |\, \langle \operatorname{Rm} (\Omega), \overline{\Omega}\rangle\ge 0, \text{ for any } \Omega, \operatorname{rank}(\Omega)=1, \Omega^2=0\}.\ee
For the other two conditions we recall two operators from the study of the Ricci flow. The first one is the $\bar{\wedge}$ operator on $A, B $, any two Hermitian symmetric transformations on $T'M$,  defined by
\begin{eqnarray*}
(A\bar{\wedge} B)_{i\bar{j}k\bar{l}}&=&A_{i\bar{j}}B_{k\bar{l}}+B_{i\bar{j}}A_{k\bar{l}}
+A_{i\bar{l}}B_{k\bar{j}}+B_{i\bar{l}}A_{k\bar{j}}\\
&=&2\langle(A\wedge \bar{B}+B\wedge \bar{A}) (e_i\wedge e_{\bar{j}}), \overline{e_l\wedge e_{\bar{k}}}\rangle\\&\quad& +2\langle (A\wedge \bar{B}+B\wedge \bar{A}) (e_k\wedge e_{\bar{j}}), \overline{e_l\wedge e_{\bar{i}}}\rangle.
\end{eqnarray*}
The resulting operator so defined is also a K\"ahler curvature operator which , in particular, satisfies the  first Bianchi identity.  Here $(A\wedge B)(X\wedge Y)=\frac{1}{2}(A(X)\wedge B(Y)+B(X)\wedge A(Y))$ as defined in \cite{Wi}. This operator is the one involved in the $\mathsf{U}(n)$-invariant irreducible decomposition of the space of the K\"ahler curvature operators. Now the condition (NQOB) is equivalent to
\be \label{NQOB}
\{\operatorname{Rm} \, |\, \langle \operatorname{Rm}, A^2\bar{\wedge}\operatorname{id}-A\bar{\wedge}A\rangle\ge 0, \text{ for all Hermitian symmetric } A\}.
\ee
For (NCF) it can be identified with that
\be \label{NCF}
\operatorname{Ric}\wedge\operatorname{id}\wedge\operatorname{id} -2\operatorname{Rm}\wedge \operatorname{id}\ge 0
\ee
on the space of $(2,1)$-forms. Here for $X\wedge Y \wedge Z$, 
$\operatorname{Rm}\wedge \operatorname{id}(X\wedge Y\wedge Z)= \operatorname{Rm} (X\wedge Y)\wedge Z- \operatorname{Rm} (X\wedge Z)\wedge Y +X\wedge  \operatorname{Rm}  (Y\wedge Z)$, $ \operatorname{Ric}\wedge\operatorname{id}\wedge\operatorname{id}$ is defined as $\left(\operatorname{Ric}\wedge\operatorname{id}\right)\wedge \operatorname{id}$. 
 
 It is known that (NB), and (NOB) are Ricci flow invariant conditions \cite{Wi}.  It would be interesting to find out about (NQOB) and (NCF). Our speculation is that condition (NCF) follows from (NQOB) and $\Ric\ge 0$. If this speculation holds (whose validity can be easily checked when $n=2$)  one can have an alternate proof of the result in Section 4.

\bibliographystyle{amsplain}

\end{document}